\documentclass[final,leqno]{siamltex}
\usepackage{amssymb}
\usepackage{graphicx}
\usepackage{caption}
\usepackage{subcaption}
\usepackage{amsmath}

\def\ae{{\omega}}

\usepackage[usenames]{color}

\title{Parameter choice strategies for least-squares approximation of noisy smooth functions on the sphere}

\author{S. V. Pereverzyev\thanks{Johann Radon Institute for Computational and Applied Mathematics, Austrian Academy of Sciences, Altenbergerstrasse 69, 4040 Linz, Austria ({\tt sergei.pereverzyev@oeaw.ac.at}, {\tt pavlo.tkachenko@oeaw.ac.at}).}
        \and I. H. Sloan\thanks{School of Mathematics and Statistics, University of New South Wales, Sydney NSW 2052, Australia ({\tt i.sloan@unsw.edu.au}).}
	\and P. Tkachenko$^*$}

\begin{document}

\maketitle

\begin{abstract}
We consider a polynomial reconstruction of smooth functions from
their noisy values at discrete nodes on the unit sphere by a variant
of the regularized least-squares method of An et al., SIAM J. Numer. Anal.
50 (2012), 1513--1534. As nodes we use the points of a
positive-weight cubature formula that is exact for all spherical
polynomials of degree up to $2M$, where $M$ is the degree of the
reconstructing polynomial. We first obtain a reconstruction error
bound in terms of the regularization parameter and the penalization
parameters in the regularization operator. Then we discuss \emph{a
priori} and \emph{a posteriori} strategies for choosing these
parameters. Finally, we give numerical examples illustrating the
theoretical results.

\end{abstract}



\pagestyle{myheadings}
\thispagestyle{plain}
\markboth{S. V. PEREVERZYEV, I. H. SLOAN, AND P. TKACHENKO}{PARAMETER CHOICE STRATEGIES}

%--------------------------

% INTRODUCTION

%-------------------------

\section{Introduction}

In recent decades methods for approximation of a continuous function $y$
on the sphere
$\mathbb{S}^2:=\left\{\mathbf{x}=(x_1,x_2,x_3)^T\in\mathbb{R}^3:\
x_1^2+x_2^2+x_3^2=1\right\}$ by means of polynomials have been discussed
by many authors (see, for example, \cite{G1997, R2003, S1979, W1981}).
Often the underlying motivation has been the need to approximate
geophysical quantities. For example, such a task appears in the satellite
gravity gradiometry problem (SGG-problem) \cite{F1999}, p. 120, 262, \cite{P2008}, in which
the task is to find a spherical harmonic representation of Earth's
gravitational potential from satellite observations.
%when the Earth's gravitational potential is represented as an $M$-degree spherical
%polynomial, where the term ``$M$-degree spherical polynomial'' means the
%restriction to $\mathbb{S}^2$ of a polynomial of degree $M$ in $x_1,x_2$,
%and $x_3$.
The present study was motivated by this example. We shall return
to it several times throughout the paper.

The mathematical problem considered in this paper is to find a polynomial
approximation to $y\in C(\mathbb{S}^2)$, given noisy data values
$y^\epsilon(\mathbf{x}_i)$ at points $\mathbf{x}_i\in\mathbb{S}^2$,
$i=1,\ldots N$, using a least-squares strategy developed in \cite{A2012}.
(In the SGG application the sphere in question is determined by the
satellite orbits.  The gravitational potential at the satellite height
is smoother than at earth's surface, a complicating feature for the
inverse problem.) We shall assume, in a slight generalization of
\cite{A2012}, that the point set
$X_N:=\{\mathbf{x}_1,\ldots,\mathbf{x}_N\}$ consists of the points of
a cubature rule which is exact for all polynomials $p\in
\mathbb{P}_{2M}$, where $\mathbb{P}_M$ is the set of all spherical
polynomials of degree less than or equal to $M$, or in other words the
restriction to $\mathbb{S}^2$ of the set of all polynomials in
$\mathbb{R}^3$ of degree less than or equal to $M$.  Thus the point set
must satisfy
\begin{equation}\label{cubature}
\forall p\in\mathbb{P}_{2M},\ \ \sum_{i=1}^N w_ip(\mathbf{x}_i)=\int_{\mathbb{S}^2}p(\mathbf{x})d\omega(\mathbf{x}),
\end{equation}
where $d \omega(\mathbf{x})$ denotes area measure on $\mathbb{S}^2$, and
$w_i, i=1,\ldots,N$ are positive cubature weights associated with the
pointset $X_N$. For sufficiently large $N$ one can find in the literature
a variety of suitable cubature formulas (see, e.g., \cite{M2001, HSW2010, LM2006,
Xu2003}). Moreover, in principle the point sets for such a rule can be
generated by selecting from any sufficiently dense set of points on the
sphere, see \cite{N2006a, LM2008, G2009}.

The strategy is to take the approximant $y_M\in\mathbb{P}_M$ to be the
minimizer of the regularized discrete least-squares problem
\begin{equation}\label{leastsquares}
y_M=\arg\min\left\{\sum_{i=1}^Nw_i(p(\mathbf{x}_i)-y^\epsilon(\mathbf{x}_i))^2+\alpha\sum_{i=1}^Nw_i(R_Mp(\mathbf{x}_i))^2,\  p\in\mathbb{P}_M\right\},
\end{equation}
where $y^\epsilon(\mathbf{x}_i):=y(\mathbf{x}_i)+\epsilon_i$
represent noisy values of a perturbed version $y^\epsilon$ of the original
function $y$ calculated at the points of $X_N$, $\alpha$ is a
regularization parameter, and $R_M:\mathbb{P}_M\rightarrow\mathbb{P}_M$ is
a linear ``penalization'' operator given by
\begin{eqnarray}\label{regularizer}
R_Mp(\mathbf{x})
&:=&\sum_{k=0}^M\beta_k\frac{2k+1}{4\pi}\int_{\mathbb{S}^2}P_k(\mathbf{x}\cdot\mathbf{z})p(\mathbf{z})d\omega(\mathbf{z})\\ \nonumber
&=&\sum_{k=0}^M\beta_k\frac{2k+1}{4\pi}\sum_{i=1}^N  w_i P_k(\mathbf{x}\cdot\mathbf{x}_i)p(\mathbf{x}_i),\ \mathbf{x}\in\mathbb{S}^2,\ p\in\mathbb{P}_M,
\end{eqnarray}
where $P_k$ is the Legendre polynomial of degree $k$, and in the last step
we used (\ref{cubature}).  Here the numbers $\beta_k, k=1,\ldots,M$ are a
non-decreasing sequence of positive parameters. With $\beta_0$ fixed
in some appropriate way, the important feature of the parameters $\beta_k$
is their rate of growth. The central task in this paper will be to assign
appropriate values for the $\beta_k$.

As pointed out in \cite{A2012}, the expression in (\ref{regularizer})
is the most general rotationally invariant expression for a linear
operator on the space $\mathbb{P}_M$.

In \cite{A2012} the point set $X_N$ was taken to be a spherical
$2M$-design, which simply means that (\ref{cubature}) must hold with equal
weights $w_i=4\pi/N$.  We gain considerable freedom in this paper by
allowing general positive weights $w_i$ in (\ref{cubature}).  The only
effective difference in the present approximation scheme is that the
least-squares problem (\ref{leastsquares}) is slightly non-standard
because of the appearance of the cubature weights $w_i$.

It was observed in numerical experiments in \cite{A2012} that a
proper choice of the penalization operator $R_M$ together with the
regularization parameter $\alpha$ can significantly improve the
approximation. However, the choice of the model parameters in
(\ref{regularizer}) was not settled, and still remains an open issue. In
our paper we will tackle this crucial question by proposing
parameter choice strategies (strategies for choosing $\beta_k$ and
$\alpha$) that allow good approximation of noisy smooth functions on the
sphere.

The paper is organized as follows: in the next section we present
necessary preliminaries, and give an explicit solution of the regularized
least-squares problem.. In Section 3 we derive theoretical error bounds
for the resulting approximation. Sections 4 and 5 discuss error bounds and
parameter choice strategies. Finally, in the last section we present some
numerical experiments that test the theoretical results from previous
sections.

\section{Preliminaries}

We introduce a real spherical harmonic basis for $\mathbb{P}_M$, see
\cite{M1966}
\[
\left\{Y_{k,j}:\ k=0,1,...,M,\ j=1,...,2k+1\right\},
\]
assumed to be orthonormal with respect to the standard $L^2$ inner
product,

\[
\left\langle f,g\right\rangle_{L^2(\mathbb{S}^2)} :=\int_{\mathbb{S}^2}f(\mathbf{x})g(\mathbf{x}) d\omega(\mathbf{x}).
\]
Then for $p\in\mathbb{P}_M$ an arbitrary spherical polynomial of
degree $\leq M$ there exists a unique vector
$\gamma=(\gamma_{k,j})\in\mathbb{R}^{(M+1)^2}$ such that
\begin{equation}\label{eq:1}
p(\mathbf{x})=\sum_{k=0}^M\sum_{j=1}^{2k+1}\gamma_{k,j}Y_{k,j}(\mathbf{x}),\ \ \mathbf{x}\in\mathbb{S}^2.
\end{equation}

The addition theorem for spherical harmonics (see \cite{M1966}),
which asserts
\begin{equation}\label{addition}
\sum_{j=1}^{2k+1}Y_{k,j}(\textbf{x})Y_{k,j}(\textbf{z}) = \frac{2k+1}{4\pi}P_k(\textbf{x}\cdot\textbf{z}), \quad \textbf{x},\textbf{z}\in\mathbb{S}^2,
\end{equation}
will play an important role.

The assumption that a function $y$ on the unit sphere is continuous
implies that $y\in L^2(\mathbb{S}^2)$, and hence that its Fourier
coefficients $\left\langle Y_{k,j},y\right\rangle_{L^2(\mathbb{S}^2)}$
with respect to the basis of spherical harmonics  are square-summable,
i.e.
\[
\sum_{k=0}^\infty\sum_{j=1}^{2k+1}\left|\left\langle Y_{k,j},y\right\rangle_{L^2(\mathbb{S}^2)}\right|^2<\infty.
	\]

To measure any additional smoothness of $y$ it is convenient to introduce
a Hilbert space $W^{\phi,\beta}$ that is especially tailored to the
particular problem, namely
\[
y\in W^{\phi,\beta}:=\left\{g: \left\|g\right\|^2_{W^{\phi,\beta}}:=\sum_{k=0}^\infty\sum_{j=1}^{2k+1}\frac{\left|\left\langle
Y_{k,j},g\right\rangle_{L^2(\mathbb{S}^2)}\right|^2}{\phi^2(\beta_k^{-2})}<\infty\right\},
\]
where $\phi$ is an non-decreasing function such that $\phi(0)=0$ and
$\beta=\left\{\beta_0,\beta_1,...,\beta_M,\ldots\right\}$ is the sequence of
coefficients appearing in the regularizer (\ref{regularizer}). In the
literature, see, e.g., \cite{L2013}, the function $\phi$ goes under the
name of \emph{smoothness index function}.

In this context the smoothness of $y$ is encoded in $\phi$ and
$\beta$. %influencing the best approximation of $y$ by spherical
%polynomials $p\in\mathbb{P}_m$.
For example, if the smoothness index function $\phi(t)$ and the sequence
$\beta=\left\{\beta_k\right\}$ increase polynomially with $t$ and $k$
such that $\phi(t)=O\left(t^{\nu_1}\right), \beta_k=O\left(k^{\nu_2}\right), \nu_1\nu_2>1/2$, then the space $W^{\phi,\beta}$ becomes a spherical Sobolev space $H_{2\nu_1\nu_2}$ (see, e.g., \cite{F1999}, p. 64), and a spherical analog of the fundamental lemma due to Sobolev (see \cite{F1999}, Lemma 2.1.5) says that $H_{2\nu_1\nu_2}$ is embedded in the space $C^{(\nu)}(\mathbb{S}^2)$ of functions, which have $\nu$ continuous derivatives on $\mathbb{S}^2, \nu<2\nu_1\nu_2-1$, and are the restrictions to $\mathbb{S}^2$ of functions satisfying the Laplace equation in the outer space of $\mathbb{S}^2$ and being regular at infinity. Then Jackson's theorem on the sphere (see \cite{R1971}, Theorem 3.3) tells
us that for $y\in W^{\phi,\beta}$, there holds

\begin{equation}\label{eq:3}
\inf_{p\in\mathbb{P}_M}\left\|y-p\right\|_{C(\mathbb{S}^2)}=O\left(M^{-\nu}\right), \ \nu<2\nu_1\nu_2-1.
\end{equation}

On the other hand, if the sequence $\beta=\left\{\beta_k\right\}$
increases exponentially then for polynomially increasing $\phi$ and $y\in
W^{\phi,\beta}$ we have
\[
\inf_{p\in\mathbb{P}_M}\left\|y-p\right\|_{C(\mathbb{S}^2)}=O\left(e^{-qM}\right),
\]
where $q$ is some positive number that does not depend on $M$.

In the error analysis later in the paper we make use of a linear
polynomial approximation that in a certain precise sense mimics best
approximation in the space of spherical polynomials of half the
degree. The approximation, studied in \cite{Mh2005, F2008, S2011}, approximates a
function $y\in C(\mathbb{S}^2)$ by  $V_M y\in\mathbb{P}_M$ defined by
\begin{eqnarray}\label{V_M}
V_M y(\mathbf{x})&:=\sum_{k=0}^M h\left(\frac{k}{M}\right)\sum_{j=1}^{2k+1}Y_{k,j}(\mathbf{x})\left\langle Y_{k,j},y\right\rangle_{L^2(\mathbb{S}^2)}\\
&=\sum_{k=0}^M h\left(\frac{k}{M}\right)\frac{2k+1}{4\pi}\int_{\mathbb{S}^2}P_k(\mathbf{x}\cdot\mathbf{z})y(\mathbf{z})d\omega(\mathbf{z}),\nonumber
\end{eqnarray}
where $h$ is a real-valued function on $\mathbb{R}^+$, called a
filter function, which satisfies
\begin{equation}\label{hspec}
h(t)\in[0,1]\,\forall\, t\in \mathbb{R}^+,\quad h(t) = \left\{
  \begin{array}{l l}
    1, & \quad t\in\left[0,1/2\right],\\
    0, & \quad t\in\left(1,\infty\right).
  \end{array} \right.
\end{equation}
It is shown in \cite{S2011} that for suitable choices of the filter $h$
(including  any filter in $C^3(\mathbb{R}^+)$, or the unique $C^1$
quadratic spline with breakpoints at 1/2, 3/4 and 1 that satisfies
(\ref{hspec})), the norm of the operator $V_M$ as an operator from
$\mathbb{P}_M$ to $C(\mathbb{S}^2)$ is bounded independently of $M$.
Since, as is easily seen, $V_M$ reproduces polynomials of degree less than
or equal to $M/2$, it follows in the usual way that
\[
\left\|y-V_My\right\|_{C(\mathbb{S}^2)}\leq c\inf_{p\in\mathbb{P}_{\left[ M/2\right]}}\left\|y-p\right\|_{C(\mathbb{S}^2)},
\]
where $\left[\cdot\right]$ denotes the floor function.  (In this paper $c$
is a generic constant, which may take different values at different
occurrences.)  In view of (\ref{eq:3}), for polynomially increasing
$\phi, \beta$ and $y\in W^{\phi,\beta}$ we have
\[
\left\|y-V_My\right\|_{C(\mathbb{S}^2)}\leq c\left[M/2\right]^{-\nu}\leq cM^{-\nu}.
\]

On the other hand, for exponentially increasing $\beta$ and polynomially
increasing $\phi$ the theory \cite{R2000} suggests taking $h(t) = 1$
for  $t\in\left[0,1\right]$ (in which case $V_M y$ is just the
$M$th-degree partial sum of the Fourier-Laplace series). Then for
$y\in W^{\phi,\beta}$ there holds
\[
\left\|y-V_My\right\|_{C(\mathbb{S}^2)}\leq c\sqrt{M}\inf_{p\in\mathbb{P}_M}\left\|y-p\right\|_{C(\mathbb{S}^2)}\leq c\sqrt{M}e^{-qM}.
\]

\section{Weighted regularized least-squares problem and its solution}

The penalization operator (\ref{regularizer}) can equivalently be written,
using the addition theorem (\ref{addition}) and (\ref{eq:1}), as
\begin{eqnarray}\label{equivalentR}
R_Mp(\mathbf{x})&=&\sum_{k=0}^M\beta_k\sum_{j=1}^{2k+1}Y_{k,j}(\mathbf{x})\left\langle Y_{k,j},p \right\rangle_{L^2(\mathbb{S}^2)}\\ 
&=&\sum_{k=0}^M\beta_k\sum_{j=1}^{2k+1}\gamma_{k,j}Y_{k,j}(\textbf{x}),\nonumber
\end{eqnarray}
allowing us to write the minimization problem as one of linear algebra.
For the noisy function $y^\epsilon$ defined on $\mathbb{S}^2$, let
$\mathbf{y^\epsilon}:=\mathbf{y^\epsilon}(X_N)$ be the column vector
\[
\mathbf{y^\epsilon}=\left[y^\epsilon(\mathbf{x}_1),...,y^\epsilon(\mathbf{x}_N)\right]^T\ \in \mathbb{R}^N,
\]
and let $\mathbf{Y_M}:=\mathbf{Y_M}(X_N)\in\mathbb{R}^{(M+1)^2\times N}$
be the matrix of spherical harmonics evaluated at the points of $X_N$.
Using this notation we can reduce the minimization problem in
(\ref{leastsquares}) to the following discrete minimization problem:

\begin{equation}\label{eq:10}
\min_{\boldsymbol{\gamma}\in\mathbb{R}^{(M+1)^2}}\left\|\mathbf{W}^{1/2}\mathbf{Y_M}^T\boldsymbol{\gamma}-\mathbf{W}^{1/2}\mathbf{y^\epsilon}\right\|^2_2+\alpha\left\|\mathbf{W}^{1/2}\mathbf{R_M}^T\boldsymbol{\gamma}\right\|^2_2,\ \alpha>0,
\end{equation}
where $\left\|\cdot\right\|_2$ is the standard Euclidean vector norm, $\mathbf{R_M}:=\mathbf{R_M}(X_N)=\mathbf{B_M}\mathbf{Y_M}\in\mathbb{R}^{(M+1)^2\times N}$, $\mathbf{B_M}$ is a positive diagonal matrix defined by

\begin{equation}\label{eq:11}
\mathbf{B_M}:=\diag(\beta_0,\underbrace{\beta_1,\beta_1,\beta_1}_{3},...,\underbrace{\beta_M,\beta_M,...,\beta_M}_{2M+1})\in\mathbb{R}^{(M+1)^2\times(M+1)^2},
\end{equation}
and $\mathbf{W}$ is a diagonal matrix of cubature weights
\[
\mathbf{W}:=\diag(w_1,...,w_N)\in\mathbb{R}^{N\times N}.
\]

The solution of (\ref{leastsquares}) can be found from the following
system of linear equations

\begin{equation}\label{eq:12}
(\mathbf{Y_M}\mathbf{W}\mathbf{Y_M}^T+\alpha\mathbf{B_M}\mathbf{Y_M}\mathbf{W}\mathbf{Y_M}^T\mathbf{B_M})\gamma=\mathbf{Y_M}\mathbf{W}\mathbf{y^\epsilon}.
\end{equation}

\begin{theorem} \label{th:solution}
Assume $y^\epsilon\in C(\mathbb{S}^2)$. Let $M>0$ be given, and let
(\ref{cubature}) hold true for the set of points $X_N$. Then (\ref{eq:12})
has the unique solution $\boldsymbol{\gamma}=(\gamma_{k,j})\in\mathbb{R}^{(M+1)^2}$,

\begin{equation}\label{eq:13}
\gamma_{k,j}=\frac{1}{1+\alpha\beta_k^2}\sum_{i=1}^N w_i Y_{k,j}(\mathbf{x}_i)y^\epsilon(\mathbf{x}_i),
\end{equation}
and the minimizer of (\ref{leastsquares}) is given by

\begin{eqnarray}\label{eq:14}
y_M(\mathbf{x})=T_{\alpha,M}^\beta y^\epsilon(\mathbf{x}):&=&\sum_{k=0}^M\sum_{j=1}^{2k+1}\frac{Y_{k,j}(\mathbf{x})}{1+\alpha\beta_k^2} \sum_{i=1}^N w_iY_{k,j}(\mathbf{x}_i)y^\epsilon(\mathbf{x}_i)\\
&=&\sum_{k=0}^M\frac{2k+1}{4\pi}\frac{1}{1+\alpha\beta_k^2} \sum_{i=1}^N w_iP_{k}(\mathbf{x}\cdot\mathbf{x}_i)y^\epsilon(\mathbf{x}_i).\nonumber
\end{eqnarray}

\end{theorem}

\begin{proof}
On using (\ref{cubature}) we have
\[
\sum_{i=1}^N w_iY_{k,j}(\mathbf{x}_i)Y_{\kappa,\iota}(\mathbf{x}_i)=
\left\langle Y_{k,j},Y_{\kappa,\iota}\right\rangle_{L^2(\mathbb{S}^2)}=\delta_{k,\kappa}\delta_{j,\iota},
\]
where $k,\kappa=0,...,M,\ j=1,...,2k+1,\ \iota=1,...,2\kappa+1$.  Thus
$\mathbf{Y_M}\mathbf{W}\mathbf{Y_M}^T$ is the identity matrix. Since
$\mathbf{B_M}$ and $\mathbf{W}$ are diagonal matrices, the solution of
(\ref{eq:12}) is given by (\ref{eq:13}) and from (\ref{eq:1}) we obtain
(\ref{eq:14}). \qquad\end{proof}

\textbf{Remark 3.1}. Note that one can also employ fast iterative algorithms for scattered least squares \cite{K2007} to find the minimizer (\ref{leastsquares}). Moreover, the evaluation of the coefficients (\ref{eq:13}) can be realized with fast spherical Fourier transform presented in \cite{K2008}.
%--------------------------

% Section

%-------------------------

\section{Error bounds}

In this section we estimate the uniform error of approximation of $y$ by
$y_M$, see (\ref{eq:14}). It is convenient here to regard
$y^\epsilon$ as a continuous function on $\mathbb{S}^2$, constructed by
some interpolation process from its values on the discrete set $X_N$.  The
operator $T_{\alpha,M}^\beta$ defined in (\ref{eq:14}) can then be
considered as an operator on the space $C(\mathbb{S}^2)$. Since
$y_M=T_{\alpha,M}^{\beta} y^\epsilon$ it is clear that
\[
y-y_M=y-T_{\alpha,M}^\beta V_M y+T_{\alpha,M}^\beta(V_M y - y + y - y^\epsilon),
\]
and hence
%$\left\|y-y_M\right\|_{C(\mathbb{S}^2)}=\left\|y-T_{\alpha,M}^{\beta}
%y^\epsilon\right\|_{C(\mathbb{S}^2)}$. It is clear that
\begin{eqnarray}\label{eq:15}
\left\|y-y_M\right\|_{C(\mathbb{S}^2)}&\leq&\left\|y-T_{\alpha,M}^{\beta} V_My\right\|_{C(\mathbb{S}^2)} \\ \nonumber
&+&\left\|T_{\alpha,M}^{\beta}\right\|_{C(\mathbb{S}^2)}\left(\left\|y-V_My\right\|_{C(\mathbb{S}^2)}+\left\|y-y^\epsilon\right\|_{C(\mathbb{S}^2)}\right),
\end{eqnarray}
where $\left\|T_{\alpha,M}^{\beta}\right\|_{C(\mathbb{S}^2)}$ is the norm of the operator $T_{\alpha,M}^{\beta}:C(\mathbb{S}^2)\rightarrow C(\mathbb{S}^2)$.

Let $\epsilon=\left[\epsilon_1,\epsilon_2,...,\epsilon_N\right] \in
\mathbb{R}^N$, and
$\left\|\epsilon\right\|_\infty=\max\left|\epsilon_i\right|$. It is
natural to assume, and from now on we shall do so, that
$\left\|y-y^\epsilon\right\|_{C(\mathbb{S}^2)}=
\left\|\epsilon\right\|_\infty$.  This means that we adopt the deterministic noise model, which allows the worst noise level at any point of $X_N$. Then it is also natural to assume
that $M$ is large enough to
ensure
$\left\|y-V_My\right\|_{C(\mathbb{S}^2)}\le\left\|\epsilon\right\|_\infty$,
since otherwise data noise is dominated by the approximation error and no
regularization is required. Then the bound (\ref{eq:15}) can be reduced to
\begin{equation}\label{eq:16}
\left\|y-y_M\right\|_{C(\mathbb{S}^2)}\leq\left\|y-T_{\alpha,M}^{\beta}V_My\right\|_{C(\mathbb{S}^2)}+2\left\|T_{\alpha,M}^{\beta}\right\|_{C(\mathbb{S}^2)}\left\|\epsilon\right\|_\infty.
\end{equation}

We will call the first term of the right-hand side in (\ref{eq:16}) the
\emph{regularization error} and the second the \emph{noise
propagation error}.

The noise propagation error can be quantified by the following result
for the norm of $T_{\alpha, M}^\beta$,
%$\left\|T_{\alpha,M}^{\beta}\right\|_{C(\mathbb{S}^2)}\left\|\epsilon\right\|_\infty$,
which is a consequence of (\ref{eq:14}).
\begin{theorem} \label{th:regerr}
Under the conditions of Theorem \ref{th:solution}

\begin{eqnarray}\label{eq:17}
\left\|T_{\alpha,M}^{\beta}\right\|_{C(\mathbb{S}^2)}&=&\max_{\mathbf{x}\in\mathbb{S}^2}\sum_{i=1}^N w_i
\left| \sum_{k=0}^M \frac{2k+1}{4\pi(1+\alpha\beta_k^2)}P_k(\mathbf{x}\cdot\mathbf{x}_i)\right|\\ \nonumber
& \le &\max_{\mathbf{x}\in\mathbb{S}^2}\sum_{i=1}^N w_i\sum_{k=0}^M\frac{2k+1}{4\pi(1+\alpha\beta_k^2)}\left|P_k(\mathbf{x}\cdot\mathbf{x}_i)\right|.
\end{eqnarray}

\end{theorem}
Theorem \ref{th:regerr} reduces to Proposition 5.1 in \cite{A2012} on
setting $w_i=4\pi/N$, but note that the result as stated in \cite{A2012}
corresponds to the upper bound in (\ref{eq:17}), and so is not
correctly stated.

Now we are going to bound the regularization error $\left\|y-T_{\alpha,M}^{\beta}V_My\right\|_{C(\mathbb{S}^2)}$. We start with the following decomposition

\begin{equation}\label{eq:18}
y-T_{\alpha,M}^{\beta}V_My=y-T_{0,M}V_My+(T_{0,M}-T_{\alpha,M}^{\beta})V_My,
\end{equation}
where $T_{0,M}$ is the so-called hyperinterpolation operator
\cite{S1995},
\begin{eqnarray}\label{eq:19}
T_{0,M}g(\mathbf{x})%&:=&\arg\min\left\{\sum_{i=1}^Nw_i(p(\mathbf{x}_i)-y(\mathbf{x}_i))^2,\  p\in\mathbb{P}_M\right\} \\ \nonumber
&=&\sum_{k=0}^M\sum_{j=1}^{2k+1}Y_{k,j}(\mathbf{x}) \sum_{i=1}^N w_iY_{k,j}(\mathbf{x}_i)g(\mathbf{x}_i)\\ \nonumber
&=&\sum_{k=0}^M \frac{2k+1}{4\pi} \sum_{i=1}^N w_i P_k(\mathbf{x}\cdot \mathbf{x}_i)g(\mathbf{x}_i).
\end{eqnarray}

From (\ref{eq:19}) and (\ref{cubature}) it immediately follows that
for any $p\in\mathbb{P}_M$ we have $T_{0,M}p=p$. Therefore,
$T_{0,M}V_My=V_My$. In view of this property and the decomposition
(\ref{eq:18}) we can derive a bound for the regularization error

\begin{eqnarray}\label{eq:20}
\left\|y-T_{\alpha,M}^{\beta}V_My\right\|_{C(\mathbb{S}^2)}&\leq& \left\|y-V_My\right\|_{C(\mathbb{S}^2)}+\left\|(T_{0,M}-T_{\alpha,M}^{\beta})V_My\right\|_{C(\mathbb{S}^2)} \\ \nonumber
&\leq& \left\|\epsilon\right\|_\infty+\left\|(T_{0,M}-T_{\alpha,M}^{\beta})V_My\right\|_{C(\mathbb{S}^2)}.
\end{eqnarray}

An estimate of the term $\left\|
(T_{0,M}-T_{\alpha,M}^{\beta})V_M\right\|_{C(\mathbb{S}^2)}$ in
(\ref{eq:20}) is given by the following theorem.

\begin{theorem} \label{th:apb}
Assume that the smoothness index function $\phi$ is such that the function
$t\rightarrow t/\phi(t)$ is monotone. Then for $y\in W^{\phi,\beta}$
there holds

\begin{equation}\label{eq:21}
\left\|(T_{0,M}-T_{\alpha,M}^{\beta})V_My\right\|_{C(\mathbb{S}^2)}\leq cM\hat{\phi}(\alpha)\left\|y\right\|_{W^{\phi,\beta}},
\end{equation}
where $\hat{\phi}(\alpha)=\phi(\alpha)$ if $t/\phi(t)$ is
non-decreasing, and $\hat{\phi}(\alpha)=\alpha$ if $t/\phi(t)$ is
non-increasing.
\end{theorem}

\begin{proof}
In view of (\ref{eq:14}), (\ref{eq:19}) and (\ref{V_M}), together with the
fact that the cubature formula in (\ref{cubature}) is exact for
$p\in\mathbb{P}_{2M}$, we may write
\begin{eqnarray*}
&&\left\|(T_{0,M}-T_{\alpha,M}^{\beta})V_My\right\|_{C(\mathbb{S}^2)}\\ \nonumber
&&=\left\|\sum_{k=0}^M\sum_{j=1}^{2k+1}Y_{k,j}(\cdot)\frac{\alpha\beta_k^2}{1+\alpha\beta_k^2}\left\langle Y_{k,j},V_M y\right\rangle_{L^2(\mathbb{S}^2)}\right\|_{C(\mathbb{S}^2)}\\ \nonumber
&&=\left\|\sum_{k=0}^M\sum_{j=1}^{2k+1}h\left(\frac{k}{M}\right)Y_{k,j}(\cdot)\frac{\alpha\beta_k^2}{1+\alpha\beta_k^2}\left\langle Y_{k,j},y\right\rangle_{L^2(\mathbb{S}^2)}\right\|_{C(\mathbb{S}^2)},
\end{eqnarray*}
where in the last step we used $\langle Y_{k,j},V_M
y\rangle_{L^2(\mathbb{S}^2)}=h(k/M) \langle
Y_{k,j},y\rangle_{L^2(\mathbb{S}^2)}$.  Now using the Nikolskii
inequality (see, e.g., \cite{N2006}, Proposition 2.5) and also
$h(k/M)\le 1$, we obtain
\begin{eqnarray*}
%&&\left\|\sum_{k=0}^M\sum_{j=1}^{2k+1}h\left(\frac{k}{M}\right)Y_{k,j}(\cdot)\frac{\alpha\beta_k^2}{1+\alpha\beta_k^2}\left\langle Y_{k,j}(\cdot),y\right\rangle_{L^2(\mathbb{S}^2)}\right\|_{C(\mathbb{S}^2)} \\
\|(T_{0,M}-T_{\alpha,M}^{\beta})V_My\|_{C(\mathbb{S}^2)}&&\leq cM\left\|\sum_{k=0}^M\sum_{j=1}^{2k+1}h\left(\frac{k}{M}\right)Y_{k,j}\frac{\alpha\beta_k^2}{1+\alpha\beta_k^2}\left\langle Y_{k,j},y\right\rangle_{L^2(\mathbb{S}^2)}\right\|_{L^2(\mathbb{S}^2)} \\
&&= cM\left(\sum_{k=0}^M\sum_{j=1}^{2k+1}h\left(\frac{k}{M}\right)^2\left(\frac{\alpha\beta_k^2}{1+\alpha\beta_k^2}\right)^2\left|\left\langle Y_{k,j},y\right\rangle_{L^2(\mathbb{S}^2)}\right|^2\right)^{1/2}\\
%\end{eqnarray*}
%Using the assumption of the theorem that $y\in W^{\phi,\beta}$ we continue
%\begin{eqnarray*}
%&&cM\left(\sum_{k=0}^M\sum_{j=1}^{2k+1}\left(\frac{\alpha\beta_k^2}{1+\alpha\beta_k^2}\right)^2\left|\left\langle Y_{k,j}(\cdot),y\right\rangle_{L^2(\mathbb{S}^2)}\right|^2\right)^{1/2} \\
&&\le cM\left(\sum_{k=0}^M\sum_{j=1}^{2k+1}\left(\frac{\alpha\beta_k^2}{1+\alpha\beta_k^2}\right)^2\phi^2(\beta_k^{-2})\frac{\left|\left\langle Y_{k,j},y\right\rangle_{L^2(\mathbb{S}^2)}\right|^2}{\phi^2(\beta_k^{-2})}\right)^{1/2} \\
&&\leq
cM\sup_{u\in[0,\beta_0^{-2}]}\left|\frac{\alpha}{\alpha+u}\phi(u)\right|\left\|y\right\|_{W^{\phi,\beta}}\leq
cM\hat{\phi}(\alpha)\left\|y\right\|_{W^{\phi,\beta}},
\end{eqnarray*}
where the last inequality follows from \cite{L2013}, Proposition 2.7.
\qquad\end{proof}

It is instructive to note that if, for example, $\phi(t)=t^\nu$, then the
function $\hat{\phi}$ defined in Theorem \ref{th:apb} is given by
\[
\hat{\phi}(\alpha)=\begin{cases}
\alpha,\quad \quad \nu\ge 1,\\
\alpha^\nu, \quad 0<\nu<1.
\end{cases}
\]
Thus the error bound in the theorem does not improve if $\phi(t)$ grows
faster than $t$.

%-------------------------

\section{Parameter choice strategies}

In this section we will be concerned with the choice of the design
parameters for the least-squares approximation $y_M$, namely the
regularization parameter $\alpha$ and the penalization parameters
$\beta_k$. In the first subsection we discuss an \emph{a priori} choice
for the penalization parameters $\beta_k$.  In the next subsection we
consider an adaptive strategy for choosing the regularization
parameter $\alpha$. In the third subsection we present an \emph{a
posteriori} choice for the penalization parameters $\beta_k$.

The choice of parameters is motivated by the error bound
(\ref{eq:16}) for $y-y_M$. From (\ref{eq:20}) and (\ref{eq:21}) it
follows that the bound (\ref{eq:16}) can be reduced to the following:
\begin{eqnarray}\label{eq:26}
\left\|y-y_M\right\|&\leq& \left\|\epsilon\right\|_\infty+cM\hat{\phi}(\alpha)\left\|y\right\|_{W^{\phi,\beta}}+2\left\|\epsilon\right\|_\infty\left\|T_{\alpha,M}^{\beta}\right\|_{C(\mathbb{S}^2)} \\ \nonumber
&\leq&cM\hat{\phi}(\alpha)\left\|y\right\|_{W^{\phi,\beta}}+c\left\|\epsilon\right\|_\infty\left\|T_{\alpha,M}^{\beta}\right\|_{C(\mathbb{S}^2)}.
\end{eqnarray}
%where actually only the second term is accessible due to (\ref{eq:17}). Observe, the first term of the right-hand side in (\ref{eq:26}) grows with the growth of $\alpha$ because $\hat{\phi}$ is an increasing function, while in view of (\ref{eq:17}) the second term decreases. This observation makes it clear that the regularization parameter $\alpha$ should be chosen to balance these two terms in (\ref{eq:26}).

\subsection{A priori choice of the penalization parameters}

For definiteness, we assume in this subsection that $\phi(t)=t$, which
means that $\hat{\phi}$ has the highest order in $\alpha$, namely
$\hat{\phi}(\alpha)=\alpha$.   The error bound (\ref{eq:26}) now provides
useful guidance in the choice of the regularization parameters $\beta_k$.
If $\beta_0$ is considered to be fixed, and we increase the rate of growth
of the $\beta_k$, then the first term on the right-hand side of the
last line of (\ref{eq:26}) will increase, while from (\ref{eq:17})
the second term has an upper bound that decreases with increasing rate of
growth of the $\beta_k$. Even more can be said: for the first term to
be finite the $W^{\phi,\beta}$ norm of $y$ must be finite, which imposes
the constraint

\begin{equation}\label{eq:22}
\sum_{k=0}^\infty\sum_{j=1}^{2k+1}\beta_k^4\left\langle Y_{k,j},y\right\rangle^2_{L^2(\mathbb{S}^2)}<\infty.
\end{equation}
To see what this condition means in a particular application, we consider
the SGG-problem mentioned in the Introduction. In this problem $y$ is
the second order radial derivative of the gravitational potential measured
pointwise at the orbital sphere of a satellite. It can be shown
\cite{F2001,L2010,S1983} that after a proper normalization of this sphere
to $\mathbb{S}^2$ we have

\begin{equation}\label{eq:23}
\left\langle Y_{k,j},y\right\rangle_{L^2(\mathbb{S}^2)}=a_kg_{k,j},
\end{equation}
where $a_k=\left(\frac{R}{\rho}\right)^k\frac{(k+1)(k+2)}{\rho^2}$, $\rho$ is the radius of the orbital sphere, $R$ is the radius of the surface of the Earth considered as a sphere, and $\left\{g_{k,j}\right\}$ is some (unknown) sequence of scaled Fourier coefficients of the gravitational potential $g$ measured at the surface of the Earth. It is well-known (see, e.g., \cite{S1983a, F2001}) that in the scale of the spherical Sobolev spaces $\left\{H_s\right\}$ mentioned above the Earth's gravitational potential has a smoothness index $s=3/2$ at least, which means that the sequence $\left\{g_{k,j}\right\}$  should satisfy the requirement

\[
\sum_{k=0}^\infty\sum_{j=1}^{2k+1}(k+1/2)^3g^2_{k,j}<\infty.
\]

In view of the last requirement, the condition (\ref{eq:22}) is
satisfied by the choice
\begin{equation}\label{eq:25}
\beta_k=a_k^{-1/2}(k+1/2)^{3/4},\ \ k=0,1,... .
\end{equation}
Of course the condition (\ref{eq:22}) will also be satisfied if the
$\beta_k$ increase more slowly, but at the likely expense of a larger
second term in the error bound (\ref{eq:26}).

Since $R<\rho$, it is clear that the $\beta_k$ given by
(\ref{eq:25}) increase exponentially. This is natural in view of the exponential decrease of the Fourier coefficients (\ref{eq:23}) of the approximated function, which implies that the exact function as measured at the satellite height is very smooth, even analytic. The regularization scheme (\ref{leastsquares}) with weights (\ref{eq:25}) will penalize the presence of oscillating coefficients with large indexes in the approximant $T_{\alpha,M}^\beta y^\epsilon$. In the last section we illustrate a
good performance of the scheme (\ref{leastsquares}) with these
penalization weights.

\subsection{Regularization parameter choice strategy}
For regularization of our problem we will implement an adaptive
regularization parameter choice strategy known as the balancing principle
(see, e.g., \cite{L2013, M2003, P2005} and references therein). In
this method the regularization parameter $\alpha$ is selected from some
finite set, say $\Delta_L:=\left\{\alpha_i=q^i\alpha_0,
i=1,2,...,L\right\}$, with $q\in (0,1)$ and $L$ large enough.

Applying the balancing principle to our problem we start with the smallest
parameter $\alpha_L$ and increase stepwise $\alpha_{i-1}=\alpha_i/q,
i=L,L-1,...,$ until $\alpha_*:=\alpha_z$ is the parameter for which

\[
\left\|T_{\alpha_z,M}^\beta y^\epsilon-T_{\alpha_{z+1},M}^\beta y^\epsilon\right\|_{C(\mathbb{S}^2)}>{\ae}\left\|\epsilon\right\|_\infty\left\|T_{\alpha_{z+1},M}^\beta\right\|_{C(\mathbb{S}^2)}
\]
for the first time. Here ${\ae}$ is a design parameter. In all our numerical tests with the balancing principle (BP) reported below in Section 6, the value of ${\ae}$ is fixed as ${\ae}=0.002$ and is data independent, while the value of the regularization parameter $\alpha_*$ chosen according to BP varies with data.  

Note that for choosing $\alpha_*$ we need only the knowledge of (\ref{eq:14}) and an
upper bound of $\left\|T_{\alpha,M}^{\beta}\right\|_{C(\mathbb{S}^2)}$
given by (\ref{eq:17}).

In the Section 6 we will present a numerical test showing a good
reconstruction of the function on the sphere from noisy observations with
the above \emph{a posteriori} regularization parameter. It is instructive to see that in all tests BP performs at the level of the ideal parameter choice $\alpha\in\Delta_L$.

\subsection{A posteriori choice of the penalization weights}

We start with the observation that the space  $\mathbb{P}_M$ of spherical
polynomials $p$ is a reproducing kernel Hilbert space (RKHS)
$\mathcal{H}$. By the Riesz representation theorem, to every RHKS
$\mathcal{H}$ there corresponds a unique symmetric positive definite
function $K:\mathbb{S}^2\times\mathbb{S}^2\rightarrow \mathbb{R}$, called
the reproducing kernel of $\mathcal{H}=\mathcal{H}_K$, that has the
following reproducing property: $p(\mathbf{x})=\left\langle
p(\cdot),K(\cdot,\mathbf{x})\right\rangle_{\mathcal{H}_K}$. A
comprehensive theory of RKHSs can be found in \cite{A1950}.

It is easy to check that the kernel

\begin{equation}\label{eq:a}
K(\mathbf{x},\mathbf{z})=\sum_{k=0}^M \beta_k^{-2}\sum_{j=1}^{2k+1} Y_{k,j}(\mathbf{x}) Y_{k,j}(\mathbf{z}),\ \mathbf{x}, \mathbf{z}\in\mathbb{S}^2
\end{equation}
has the above mentioned reproducing property if the inner product in  $\mathbb{P}_M$ is defined as follows

\[
\left\langle f,g \right\rangle_{\mathcal{H}_K}=\sum_{k=0}^M \beta_k^2\sum_{j=1}^{2k+1} \left\langle Y_{k,j},f\right\rangle_{L^2(\mathbb{S}^2)} \left\langle Y_{k,j},g\right\rangle_{L^2(\mathbb{S}^2)}.
\]
Indeed, for $p\in\mathbb{P}_M$ we find

\begin{eqnarray*}
\left\langle p(\cdot),K(\mathbf{x},\cdot) \right\rangle_{\mathcal{H}_K}&=&\sum_{k=0}^M \beta_k^2\sum_{j=1}^{2k+1} \left\langle Y_{k,j},p\right\rangle_{L^2(\mathbb{S}^2)} \left\langle Y_{k,j},K(\mathbf{x},\cdot)\right\rangle_{L^2(\mathbb{S}^2)} \\ \nonumber
%&=&\sum_{k=0}^M \beta_k^2\sum_{j=1}^{2k+1} \left\langle Y_{k,j},p\right\rangle_{L^2(\mathbb{S}^2)} \left\langle Y_{k,j},\sum_{l=0}^M \beta_l^{-2}\sum_{i=1}^{2l+1} Y_{l,i}(\cdot) Y_{l,i}(\mathbf{x})\right\rangle_{L^2(\mathbb{S}^2)} \\ \nonumber
&=&\sum_{k=0}^M \beta_k^2\sum_{j=1}^{2k+1}\left\langle Y_{k,j},p\right\rangle_{L^2(\mathbb{S}^2)}\beta_k^{-2}Y_{k,j}(\mathbf{x})=p(\mathbf{x}).
\end{eqnarray*}
%in the second last step using $\langle
%Y_{k,j},Y_{\ell,i}\rangle_{L^2(\mathbb{S}^2)} =
%\delta_{k,\ell}\delta_{j,i}$.

In this RKHS setting the spherical polynomial $y_M=y_M(N,K,\alpha)$
defined by (\ref{leastsquares}) also can be seen, using the addition
theorem and (\ref{cubature}), as the minimizer of the following
quadratic functional

\begin{equation}\label{eq:b}
T_\alpha(N,K;p)=\sum_{i=1}^Nw_i(p(\mathbf{x}_i)-y^\epsilon(\mathbf{x}_i))^2+\alpha\left\|p\right\|^2_{\mathcal{H}_K},\  p\in\mathbb{P}_M,
\end{equation}
which makes (\ref{eq:a}) a natural way of defining the reproducing
kernel in this context.

At this point the problem of the choice of the penalization weights
$\left\{\beta_k\right\}$ is transformed into that of selecting a
kernel $K$ from the set $\mathcal{K}$ of kernels of the form (\ref{eq:a}).

In the literature there are several methods for choosing a kernel from the
available set of kernels (see, e.g., \cite{M2005, N2011}, and references
therein). For example, in \cite{M2005} the authors suggest selecting
a kernel by minimizing the  value of the functional (\ref{eq:b}) evaluated
at its minimizer $y_M$. In the present context, according to \cite{M2005},
the kernel $K=K_*$ of choice is given as

\[
K_*=\arg \min\left\{T_\alpha(N,K;y_M(N,K,\alpha)),\ K\in\mathcal{K}\right\}.
	\] Note that such $K_*$ depends on the value of the regularization
parameter $\alpha$. Therefore, the approach \cite{M2005} can be realized
only for an \emph{a priori} known $\alpha$. However, in practice we
are not provided with this knowledge, and have to use \emph{a posteriori}
regularization parameter choice strategies (for example, the balancing
principle described in Subsection 4.1). Thus, in practice we are dealing
with kernel dependent regularization parameter $\alpha=\alpha(K)$.

This situation has been discussed in \cite{N2011}. In the present context the kernel choice suggested in \cite{N2011} can be written as follows

\begin{equation}\label{eq:c}
K_+=\arg \min\left\{T_{\alpha(K)}(N,K;y_M(N,K,\alpha(K))),\ K\in\mathcal{K}\right\}.
\end{equation}
The existence of such $K_+$ has been proved in \cite{N2011} under rather general assumptions on the set of admissible kernels $\mathcal{K}$ and regularization parameter choice strategy $\alpha=\alpha(K)$.

From a practical point of view, it is a challenging issue to use the
strategy (\ref{eq:c}) in our case  because one has to minimize a function
depending on $M+1$ unknown penalization weights $\beta_k$. Therefore, it
is natural to reduce the complexity of the model before applying the
strategy from  \cite{N2011}.

For example, one may parametrize $\left\{\beta_k\right\}$ as follows: $\beta_k^2=e^{\lambda_1(k+1)}(k+1)^{\lambda_2}, \lambda_1,\lambda_2\geq 0$. In other words, in (\ref{eq:c}) the set of kernels $\mathcal{K}$ consists of the functions

\begin{equation}\label{eq:d}
K(t,\tau)=\sum_{k=0}^M e^{-\lambda_1(k+1)}(k+1)^{-\lambda_2}\sum_{j=1}^{2k+1} Y_{k,j}(t) Y_{k,j}(\tau),\ t,\tau\in\mathbb{S}^2.
\end{equation}

Then the kernel $K_+$ can be found by minimizing a function of two
variables $\lambda_1,\lambda_2$. In the last section we will illustrate
such a reduced approach by a numerical test showing good performance of
the scheme (\ref{leastsquares}) with \emph{a posteriori} chosen penalization
weights.

%--------------------------

% Section

%-------------------------

\section{Numerical examples}

In this section we present some numerical experiments to verify the
analysis from the previous sections. Note that we work not with real
data but with artificially generated ones. In all our experiments we
follow \cite{G2009, N2013} and assume that the set of points $X_N$ is the
set of Gauss-Legendre points, for which the positive
quadrature weights are known analytically. The number of points in this
case is $N=2(M+1)^2$, and the corresponding cubature formula
(\ref{cubature}) is indeed exact for all spherical polynomials of
degree $2M$. In all our experiments $M=30$.

Note that in real applications the spherical polynomials of much higher degree are used \cite{P2008}. Moreover, the Gauss-Legendre points are known to have the drawback of having too many points concentrated at the poles, making it not suitable for real satellite data. In our experiments below we use the Gauss-Legendre points and polynomials of modest degree only for illustration purposes and as a proof of concept. At the same time, we note that even for the case $M=30$ the corresponding discrete problem is rather ill-conditioned and, thus, should be treated with a regularization (see Figure \ref{fig:1} and the discussion below).

We start with an experiment illustrating that a proper choice of the
penalization weights $\beta_0,...,\beta_M$ is crucial for the
approximation of functions on the sphere. Consider again the SGG-problem
corresponding to (\ref{eq:23}). Note that for $k=1,2,...,30$ the values
$a_k=\left(\frac{R}{\rho}\right)^k \frac{(k+1)(k+2)}{\rho^2}$ in
(\ref{eq:23}) are increasing, and so, they do not exhibit a typical
behavior of the singular values of the compact operators. This effect is
well-known (see, e.g, \cite{F1999}, Fig. 4.2.3, p. 280).

Therefore, to mimic the SGG-problem for $M=30$ one usually omits the factor $\frac{(k+1)(k+2)}{\rho^2}$ (see, e.g., \cite{B2007}). In this case the decay character of the
coefficients $a_k$ in (\ref{eq:23}) can be modeled, for example, as

\[
a_k=(1.2)^{-k},\ k=0,1,...,M.
	\]
We conduct our first experiment in the
following way. First we generate a spherical function

\[
y=y(\mathbf{x})=\sum_{k=0}^M (1.2)^{-k}\sum_{j=1}^{2k+1}g_{k,j}\frac{1}{\rho}Y_{k,j}(\mathbf{x}),\ \ \mathbf{x}\in\mathbb{S}^2,
\]
where $g_{k,j}=(k+1/2)^{-3/2}x_{k,j}, \ k=0,...,M, \ j=1,...2k+1$, and
$x_{k,j}$ are random numbers uniformly distributed on
$\left[0,1\right]$. The blurred spherical function $y^\epsilon$ is
simulated by adding a random point-wise noise to the values of the initial
function $y$ at the point set $X_N$. The simulated noise values are given as the components of a random vector $0.05\epsilon / \left\|\epsilon\right\|_\infty$, where $\epsilon=\left[\epsilon_1,\epsilon_2,...,\epsilon_N\right]$, and $\epsilon_i$ are uniformly distributed on $[-1,1]$. To mimic the SGG-problem we reconstruct the
vector $g=(g_{k,j})$ by $g^{\alpha,M}=(g_{k,j}^{\alpha,M})$, where
$g_{k,j}^{\alpha,M}=a_k^{-1}\gamma_{k,j}$, and $\gamma_{k,j}$ are given by
(\ref{eq:13}). 	
\begin{figure*}
  \includegraphics[width=\textwidth]{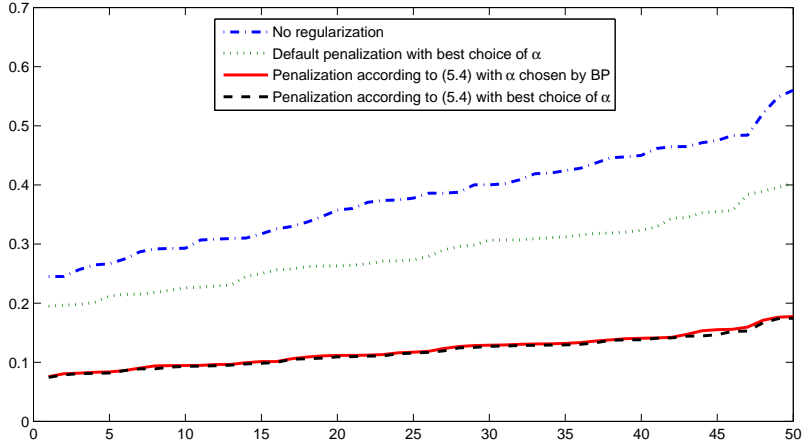}
\caption{Numerical illustration. The figure presents relative errors for 50 simulations of the data. The errors are plotted in ascending order for each of the discussed methods. Note that two bottom curves corresponding to penalization according to (\ref{eq:25}) nearly overlap.} \label{fig:1}
\end{figure*}	

To assess the obtained results and compare the performance of the considered schemes we measure the relative error

\[
\frac{\left\|g-g^{\alpha,M}\right\|_2}{\left\|g\right\|_2}.
	\]

The results are displayed in Figure \ref{fig:1},where along the vertical axis we plot the relative errors in solving the problem with one of 50 simulated data. The relative errors are plotted in ascending order for each of four methods: a straightforward least-squares fit to noisy data without any regularization, the regularization with the penalization weights (\ref{eq:25}) and $\alpha$ chosen according to the balancing principle (BP) from $\Delta_{60}=\left\{\alpha_i=8\cdot q^i\  i=1,2,\ldots,60\right\}, q=0.8$, the regularization with default penalization weights $\beta_k=1, k=0,1,\ldots,M,$ and the best $\alpha\in\Delta_{60}$, the regularization with the penalization weights (\ref{eq:25}) and the best $\alpha\in\Delta_{60}$. Thus, in the latter two cases the choice of the regularization parameter $\alpha$ for both schemes was
made to achieve the best possible performance of each method. As it can be seen from Figure \ref{fig:1} the balancing principle (BP) performs at the level of the ideal parameter choice strategy.

From Figure \ref{fig:1} one can also conclude that the proper choice of the penalization weights according to the proposed \emph{a priori} recipe can significantly improve the accuracy of the reconstruction. Moreover, Figure \ref{fig:1} shows that a straightforward least-squares fit to noisy data without regularization leads to the relative error that is about 2-3 times larger than that after a regularization. This confirms that in the considered experiment we are really dealing with a rather ill-conditioned problem.

In our second experiment we again confirm that the balancing principle gives a value of the regularization parameter $\alpha_*$ that is competitive with the best value manually found in \cite{A2012}. We choose the regularization parameter from the same geometric sequence $\Delta_{60}$ and use the same value of the design parameter $\ae=0.002$ in BP.

\begin{figure}
        \centering
        \begin{subfigure}[b]{0.45\linewidth}
                \includegraphics[width=\textwidth]{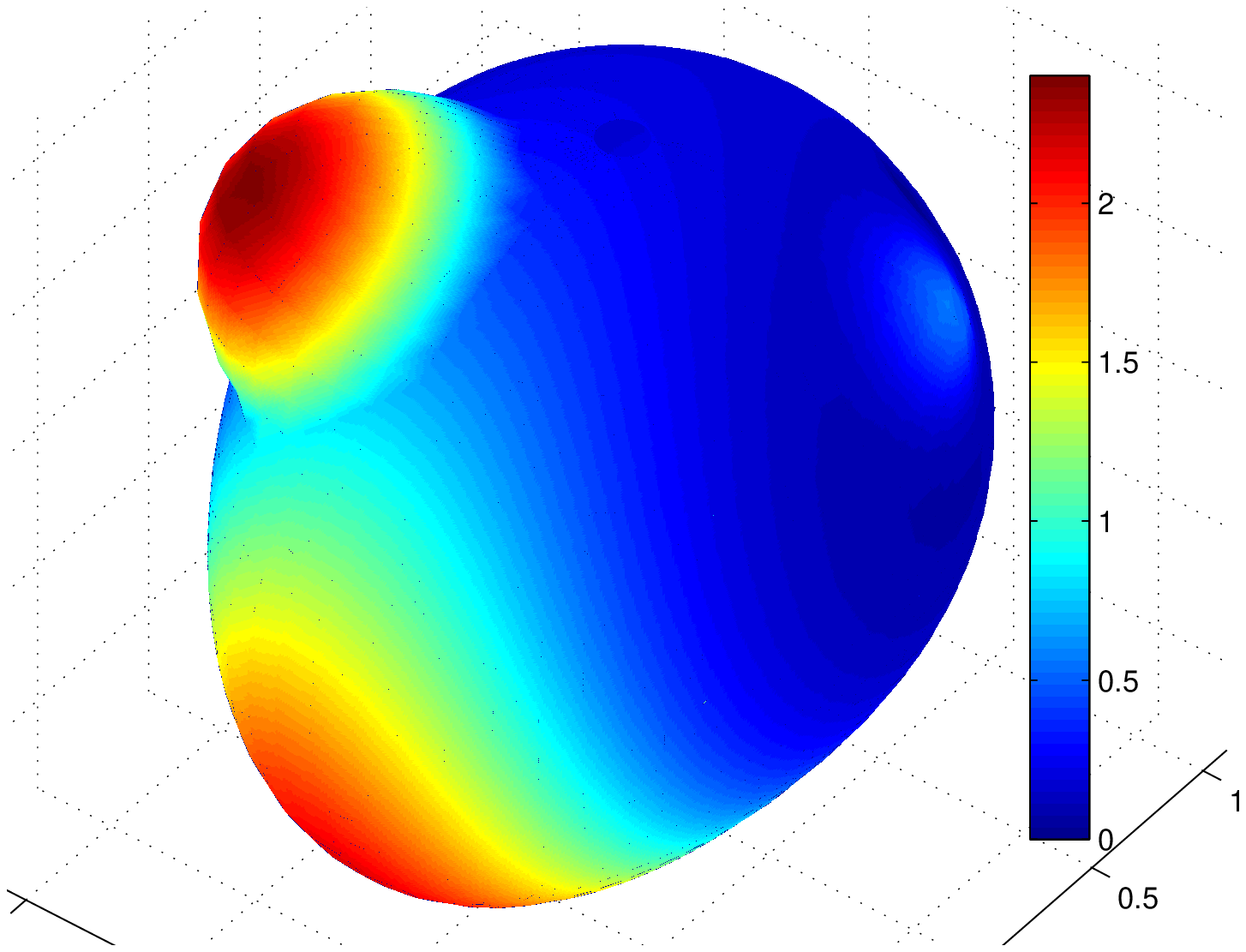}
                \caption{$y$}
                \label{fig:fr1}
        \end{subfigure}%
      \quad
        \begin{subfigure}[b]{0.45\linewidth}
                \includegraphics[width=\textwidth]{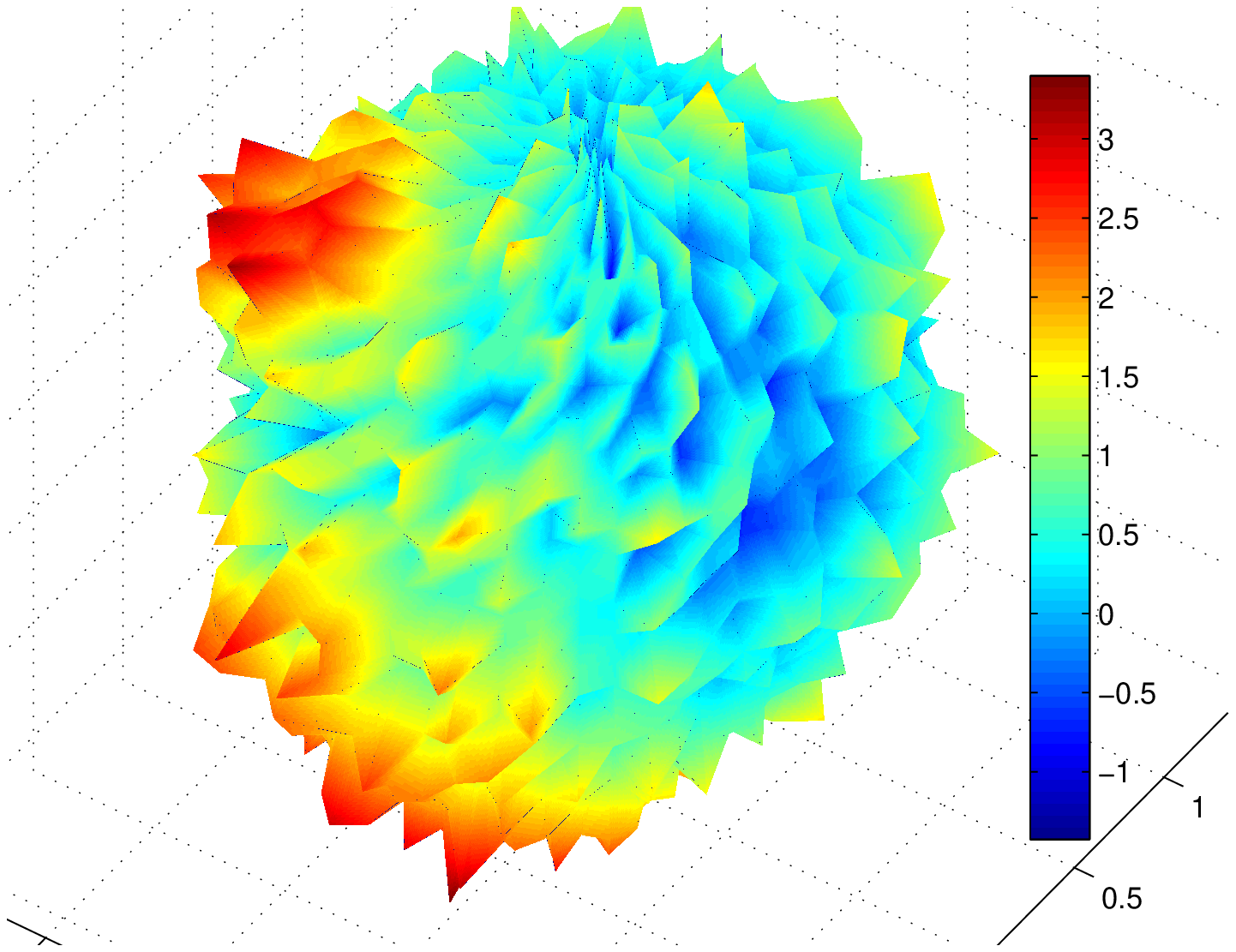}
                \caption{$y$ with $N(0,0.25)$ noise}
                \label{fig:fr2}
        \end{subfigure}
    \\
        \begin{subfigure}[b]{0.5\linewidth}
                \includegraphics[width=\textwidth]{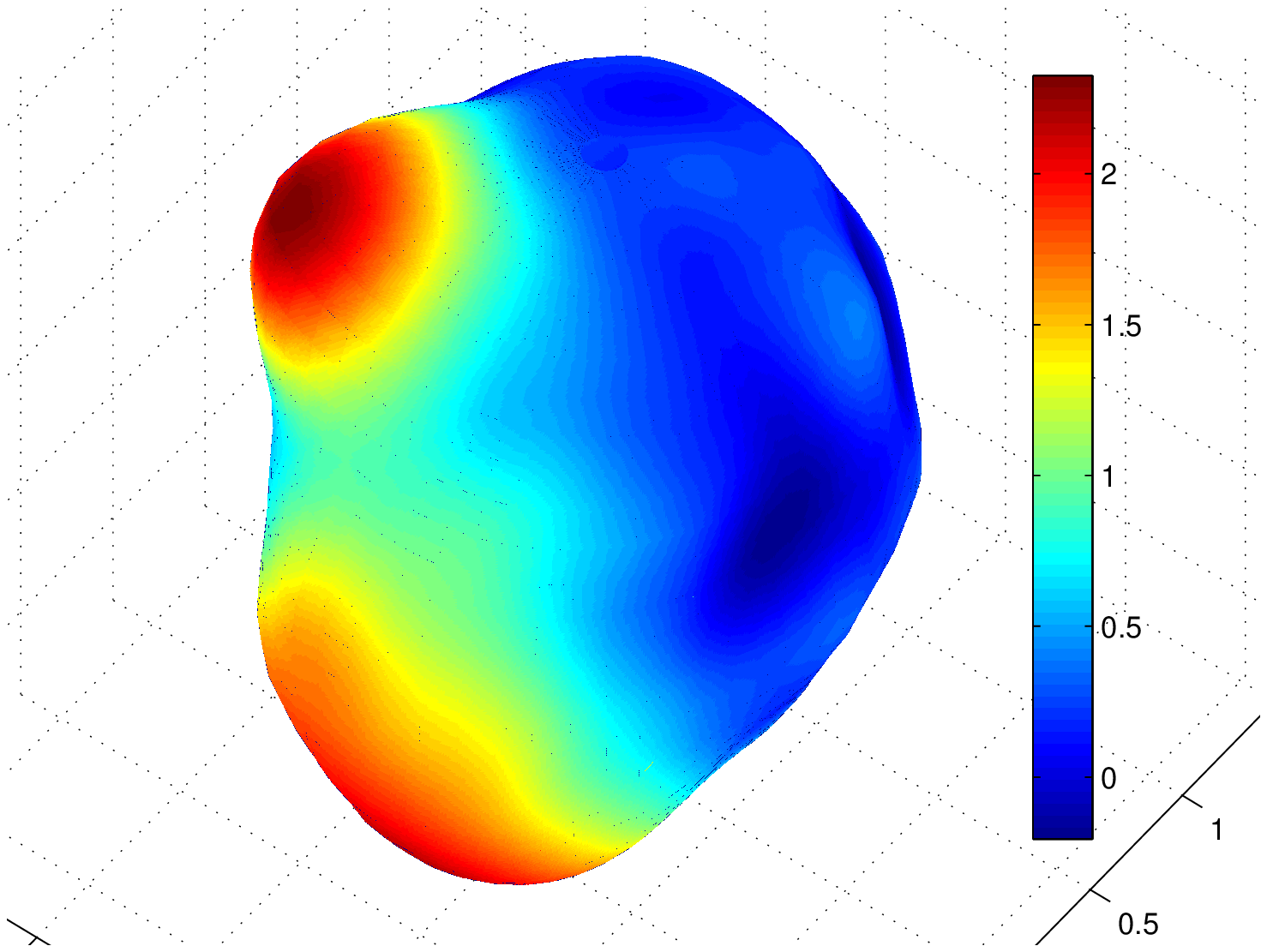}
                \caption{$T_{\alpha_*,M}y^\epsilon$}
                \label{fig:fr3}
        \end{subfigure}
        \caption{Franke function recovery}\label{fig:fr}
\end{figure}

Similarly to \cite{A2012}, as a test function $y$ we take the sum of
the Franke function $y_1$ modified by Renka \cite{R1988} (p.146) and a
function $y_{\rm{cap}}$ \cite{W1992}, namely
$y=y_1+y_{\rm{cap}}$ with

\begin{eqnarray}\label{eq:e}
y_1(x_1,x_2,x_3)&=&0.75e^{-(9x_1-2)^2/4-(9x_2-2)^2/4-(9x_3-2)^2/4} \\ \nonumber
&+&0.75e^{-(9x_1+1)^2/49-(9x_2+1)/49-(9x_3+1)/10} \\ \nonumber
&+&0.5e^{-(9x_1-7)^2/4-(9x_2-3)^2/4-(9x_3-5)^2/4}\\ \nonumber
&-&0.2e^{-(9x_1-4)^2-(9x_2-7)^2-(9x_3-5)^2},\ (x_1,x_2,x_3)\in\mathbb{S}^2,
\end{eqnarray}
and

\begin{equation}\label{eq:f}
y_{\rm{cap}}(\mathbf{x}) = \left\{
  \begin{array}{l l}
    2\cos\left(\pi\arccos(\mathbf{x}_c\cdot\mathbf{x})\right), & \quad \mathbf{x}_c\cdot\mathbf{x}\geq\cos(0.5),\\
    0, & \quad \rm{otherwise},
  \end{array} \right.
\end{equation}
where
$\mathbf{x}_c=\left(-\frac{1}{2},-\frac{1}{2},\frac{1}{\sqrt{2}}\right)^T$
and $(\cdot)$ defines the dot product of two vectors. The function $y$
was then contaminated by noise,  taking for the noise
$\epsilon(\mathbf{x})$ at each $\mathbf{x}\in X_N$ an independent sample
of a normal random variable with mean 0 and standard deviation
$\sigma=0.5$.

Figure \ref{fig:fr1} illustrates the function $y$, while Figure
\ref{fig:fr2} shows the blurred function
$y^\epsilon(\mathbf{x})=y(\mathbf{x})+\epsilon(\mathbf{x})$.

For the reconstruction, following \cite{A2012} we choose a
Laplace-Beltrami penalization operator that corresponds to the matrix

\[
\mathbf{B_M}:=\diag(0,\underbrace{4,4,4}_{3},...,\underbrace{\left(M(M+1)\right)^2,...,\left(M(M+1)\right)^2}_{2M+1})\in\mathbb{R}^{(M+1)^2\times(M+1)^2}.
\]

Figure \ref{fig:fr3} illustrates the reconstructed function
$T_{\alpha_*,M}^\beta y^\epsilon$. The regularization parameter $\alpha_*$
was obtained according to the balancing principle described above. We
found automatically the regularization parameter $\alpha_*=1.42\cdot
10^{-4}$ which agrees well with the value $10^{-4}$ from \cite{A2012}
obtained manually.

\begin{figure*}
  \includegraphics[width=\textwidth]{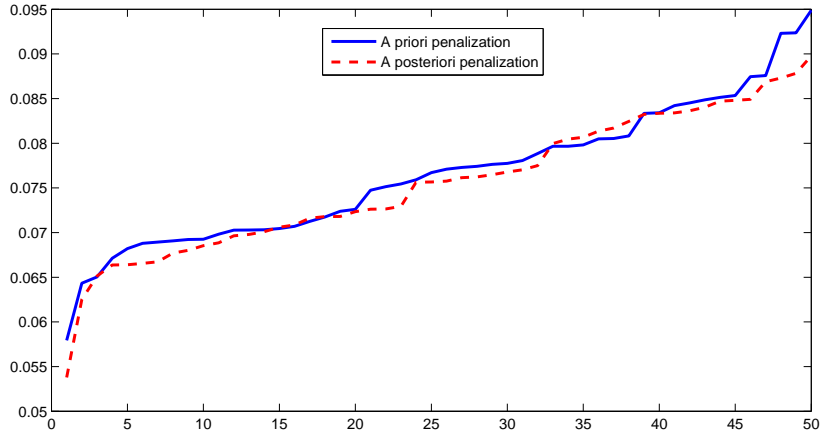}
\caption{Numerical illustration. The figure presents relative errors for 50 simulations of the data. The errors are plotted in ascending order for each of the discussed methods.} \label{fig:3}
\end{figure*}	

In our last experiment we will illustrate an application of the \emph{a
posteriori} rule (\ref{eq:c}) for choosing the penalization weights. As a
test function $y^\epsilon$ we again consider the blurred function from the
previous example, where we used the \emph{a priori} chosen penalization
weights $\beta_k=k(k+1)$ corresponding to Laplace-Beltrami operator. Now
we are going to estimate the penalization weights using the \emph{a
posteriori} strategy described in Subsection 4.3.

Recall that we are looking for the minimizer (\ref{eq:c}) among the set of
admissible kernels $\mathcal{K}$ consisting of the functions (\ref{eq:d}).
This approach allows us to take into account an exponential, as
well as a polynomial growth of $\beta_k$.

To find an approximate minimizer of (\ref{eq:c}) we have implemented
the Random Search method \cite{M1965} over the set of parameters
$(\lambda_1,\lambda_2)\in \left[0,5\right]\times\left[0,5\right]$. The
method was implemented 10 times, and in each implementation 10 random
steps have been performed. Then the mean values of the parameters
$\lambda_1, \lambda_2$ appearing after each implementation of the
Random Search method have been taken as an approximate minimum point.
As the result, the values $\lambda_1=0.32, \lambda_2=1.9$ have been
obtained.

Figure \ref{fig:3} displays the relative errors in solving the problem
(\ref{eq:e}), (\ref{eq:f}) with one of 50 simulated noisy data, for each
of two methods: regularization with the penalization weights
$\beta_k=k(k+1)$, and regularization with \emph{a posteriori} chosen
weights.

From Figure \ref{fig:3} we see that the choice of the penalization
weights according to the proposed \emph{a posteriori} choice rule can
improve the accuracy of the reconstruction.

\section*{Acknowledgments}
The first and the third authors are supported by the Austrian Fonds Zur
Forderung der Wissenschaftlichen Forschung (FWF), grant P25424. The work
was initiated when the second author visited Johann Radon Institute for
Computational and Applied Mathematics (RICAM) within the Special Semester
on Applications of Algebra and Number Theory.  The second author
acknowledges the support of the Australian Research Council.

\end{document}